\documentclass[twoside,12pt]{article}

\usepackage{amsmath}
\usepackage{amssymb}
\usepackage{amsthm}
\usepackage{fancyhdr}

\newtheorem{theorem}{Theorem}
\newtheorem{lemma}{Lemma}
\newtheorem{remark}{Remark}

\numberwithin{equation}{section}

\begin{document}

\pagestyle{fancyplain}
\lhead[\thepage]{}
\chead[Thomas Kecker]{ \fancyplain{Accepted for publication in the Journal d'Analyse Math\'ematique}{Polynomial Hamiltonian Systems}}
\rhead[]{\fancyplain{}{\thepage}}
\cfoot{\fancyplain{\thepage}{}}

\title{Polynomial Hamiltonian Systems with Movable Algebraic Singularities}

\author{Thomas Kecker}
\date{}
\maketitle

\begin{abstract}
\noindent The singularity structure of solutions of a class of Hamiltonian systems of ordinary differential equations in two dependent variables is studied. It is shown that for any solution, all movable singularities, obtained by analytic continuation along a rectifiable curve, are at most algebraic branch points.
\end{abstract}

\section{Introduction}
\noindent Singularities of solutions of ordinary differential equations can be classed as being either fixed or movable. The set of fixed singularities consists of the points in the complex plane where the equation itself becomes singular in some sense. All other singularities of a solution are called movable as their positions vary with the initial conditions of the equation. 

In the 1900's P. Painlev\'e \cite{painleve1} classified all rational second-order equations
\begin{equation}
\label{secondorder}
y'' = R(z,y,y'),
\end{equation} 
for which all solutions are single-valued around their movable singularities, a property known as the Painlev\'e property. This classification, with some errors and gaps which have been filled by R. Fuchs, B. Gambier and others, led to the discovery of six non-linear equations now known as the six Painlev\'e equations. The result of the classification is that the solution of any equation of the form (\ref{secondorder}) that has the Painlev\'e property can be expressed by the solutions of classically known function (e.g. elliptic functions, the solutions of linear differential equations or equations that are solvable by quadrature), and the solutions of the six non-linear Painlev\'e equations. Interestingly, to each Painlev\'e equation is associated an equivalent Hamiltonian system
\begin{equation}
\label{psystem}
\begin{aligned}
\frac{dq}{dz} =& \frac{\partial H_J}{\partial p} \\
\frac{dp}{dz} =& -\frac{\partial H_J}{\partial q},
\end{aligned}
\end{equation}
with Hamiltonians $H_J(z,p,q)$, $J=1,\dots,6$, polynomial in $p$ and $q$. These were already noticed by J. Malmquist \cite{malmquist2} and later extensively studied by K. Okamoto \cite{okamoto1,okamoto2,okamoto3,okamoto4}.

A classification of systems of equations in two variables that possess the Painlev\'e property,
\begin{equation}
\label{generalsystem}
\begin{aligned}
y_1' =& P(z,y_1,y_2) \\
y_2' =& Q(z,y_1,y_2),
\end{aligned}
\end{equation}
was given by R. Garnier \cite{garnier1} in the autonomous case where $P=P(y_1,y_2)$ and $Q=Q(y_1,y_2)$ are homogeneous rational functions of $y_1$ and $y_2$. The solutions in this case are given by elliptic functions or a combination of rational and exponential functions.  J. Goffar-Lombet \cite{goffar1} classified those systems (\ref{generalsystem}) with the Painlev\'e property where $P$ and $Q$ are certain polynomials of degree at most $3$ in $y_1$ and $y_2$ with $z$-dependent analytic coefficients and showed that the solutions can be given in terms of classically known functions and Painlev\'e transcendents. T. Kimura and T. Matuda \cite{kimura1} extended this result to the case where the degrees of $P$ and $Q$ are less or equal to 5. They conjecture that any system (\ref{generalsystem}) with the Painlev\'e property is equivalent to one of the systems (\ref{psystem}). 

Lifting the restriction of the Painlev\'e property, a number of articles have studied classes of second-order differential equations for which all movable singularities are at most algebraic branch points. In \cite{shimomura2,shimomura3}, S. Shimomura considered the equations
\begin{equation*}
\begin{aligned}
y'' =& \frac{2(2k+1)}{(2k-1)^2} y^{2k} + z, \quad k \in \mathbf{N}, \\
y'' =& \frac{k+1}{k^2}y^{2k+1} + zy + \alpha, \quad k \in \mathbf{N} \setminus \{2\},
\end{aligned}
\end{equation*}
of $P_I$-type and $P_{II}$-type, respectively, for which he showed that all movable singularities that can be reached by analytic continuation along a rectifiable curve are algebraic branch points. More generally, G. Filipuk and R. Halburd \cite{halburd1} have studied equations of the form
\begin{equation}
\label{filipukhalburd}
y'' = P(z,y),
\end{equation}
where $P$ is a polynomial in $y$ with analytic coefficients in $z$ that satisfy one or two differential relations known as resonance conditions (these are equivalent to the existence of certain formal algebraic series solutions of (\ref{filipukhalburd})). Equation (\ref{filipukhalburd}) can be seen as a Hamiltonian system with Hamiltonian
\begin{equation*}
H(z,y_1,y_2) = \frac{1}{2} y_2^2 - \hat{P}(z,y_1),
\end{equation*}
where $\hat{P}$ is a polynomial with $\hat{P}_y = P$ and we have the correspondence $y=y_1$, $y'=y_2$. In this article we consider a much more general class of Hamiltonian systems in two variables. The way we prove that any movable singularity of a solution that can be reached by analytic continuation along a finite length curve is an algebraic branch point is based on a method used in certain proofs of the Painlev\'e property for the Painlev\'e equations. In particular we mention the proofs in \cite{hinkkanen1}, \cite{hukuhara1} and \cite{shimomura1}, see also the book \cite{gromak}.

\section{Hamiltonian systems in two variables}
\label{hamiltoniansystem}
\noindent Consider the Hamiltonian system given by
\begin{equation}
\label{Hamiltonian}
H(z,y_1,y_2) = \alpha_{M+1,0}(z) y_1^{M+1} + \alpha_{0,N+1}(z) y_2^{N+1} + \sum_{(i,j) \in I} \alpha_{ij}(z) y_1^i y_2^j,
\end{equation}
where the set of indices $I$ is defined by
\begin{equation}
\label{indexcondition}
I = \{ (i,j) \in \mathbf{N}^2 : i(N+1) + j(M+1) < (N+1)(M+1)\},
\end{equation} 
and $\alpha_{ij}(z)$, $(i,j) \in I \cup \{(M+1,0),(0,N+1)\}$, are analytic functions in some common domain $\Omega \subset \mathbf{C}$.
The Hamiltonian equations are given by
\begin{equation}
\label{hameqns}
\begin{aligned}
y_1' =& (N+1)\alpha_{0,N+1}(z) y_2^N+ \sum_{(i,j) \in I} j \alpha_{ij}(z) y_1^i y_2^{j-1} \\
y_2' =& -(M+1)\alpha_{M+1,0}(z) y_1^M - \sum_{(i,j) \in I} i \alpha_{ij}(z) y_1^{i-1} y_2^j.
\end{aligned}
\end{equation}
The set $I$ is chosen so that $y_2^N$ and $y_1^M$ will turn out to be the dominant terms on the right hand sides of the system (\ref{hameqns}) in the vicinity of any singularity $z_\infty$ for which $\alpha_{M+1,0}(z_\infty),\alpha_{0,N+1}(z_\infty) \neq 0$. We define the set $\Phi = \{ z_0 \in \Omega | \alpha_{M+1,0}(z_0) = 0 \} \cup \{ z_0 \in \Omega | \alpha_{0,N+1}(z_0) = 0 \}$. A singularity $z_\infty$ of some solution $(y_1(z),y_2(z))$ of the system (\ref{hameqns}) is called \textit{fixed} if $z_\infty \in \Phi$. A singularity $z_\infty \notin \Phi$ of a solution of (\ref{hameqns}) is called \textit{movable}. 
Intuitively, the position of a movable singularity changes when the initial conditions of the system of differential equations are varied, whereas the fixed singularities are determined by the equation itself. A more general definition of fixed and movable singularities for first-order systems of differential equations can be found in \cite{murata1}, see also \cite{kimura2} for the case of second-order differential equations.

To determine possible leading order behaviours of a solution $(y_1,y_2)$ of (\ref{hamsystem}) about a movable singularity $z_\infty$ suppose that
\begin{equation*}
y_1 \sim c_p (z-z_\infty)^p, \qquad y_2 \sim c_q (z-z_\infty)^q.
\end{equation*} 
Assuming that the leading order terms of the right hand side of (\ref{hameqns}) are $y_2^N$ and $y_1^M$, respectively, we must have
\begin{equation*}
p-1 = Nq, \quad q-1 = Mp \quad  \Longrightarrow \quad p = -\frac{N+1}{MN-1}, \quad q = -\frac{M+1}{MN-1}.
\end{equation*}
In order for all solutions to have movable algebraic branch points a necessary condition is the existence of certain formal algebraic series solutions of (\ref{hameqns}), 
\begin{equation}
\label{formalseries}
y_1(z) = \sum_{k=-N-1}^\infty c_{1,k} (z-z_0)^{\frac{k}{MN-1}}, \quad y_2(z) = \sum_{k=-M-1}^\infty c_{2,k} (z-z_0)^{\frac{k}{MN-1}},
\end{equation}
about any point $z_0 \in \Omega \setminus \Phi$. The main result of this article is that the existence of a certain number of such formal series solutions is also sufficient for every movable singularity of a solution of (\ref{hameqns}) to be of this form. This result should be compared to the fact that for an ODE, passing the Painlev\'e test is not equivalent to having the Painlev\'e property. The existence of the series solutions is equivalent to a number of differential relations between the coefficient functions $\alpha_{ij}(z)$, $(i,j) \in I$, of the Hamiltonian $H$, known as resonance conditions, which can be calculated algorithmically. They arise from the fact that if one inserts the series (\ref{formalseries}) into the system (\ref{hameqns}) and tries to recursively determine the coefficients $c_{1,k}, k=-N-1,-N,\dots$ and $c_{2,k}, k=-M-1,-M,\dots$, the recursion breaks down at certain stages known as resonances and one is left with identities that need to be satisfied identically, leaving one coefficient at the resonance arbitrary.

\begin{theorem}
\label{maintheo}
Suppose that at every point $z_0 \in \Omega \setminus \Phi$ the Hamiltonian system (\ref{hameqns}) admits formal series solutions of the form (\ref{formalseries}) for every pair of values $(c_{1,-N-1},c_{2,-M-1})$ satisfying 
\begin{equation*}
\begin{aligned}
c_{1,-N-1}^{MN-1} &= -\left( \alpha_{0,N+1}(z_0) \alpha_{M+1,0}(z_0)^N (MN-1)^{N+1}\right)^{-1}, \\ 
c_{2,-M-1} &= (MN-1) \alpha_{M+1,0}(z_0) c_{1,-N-1}^M.
\end{aligned}
\end{equation*}
Let $\gamma \subset \Omega$ be a finite length curve with endpoint $z_\infty \in \Omega \setminus \Phi$ such that a solution $(y_1,y_2)$ can be analytically continued along $\gamma$ up to, but not including $z_\infty$. Then the solution is represented by series (\ref{formalseries}) at $z_0 = z_\infty$,
\begin{equation}
\label{solutionexpansion}
\begin{aligned}
y_1(z) =& \sum_{k=-\frac{N+1}{d}}^\infty C_{1,k} (z-z_\infty)^{\frac{kd}{MN-1}}, \\
y_2(z) =& \sum_{k=-\frac{M+1}{d}}^\infty C_{2,k} (z-z_\infty)^{\frac{kd}{MN-1}},
\end{aligned}
\end{equation}
where $d = \gcd\{M+1,N+1,MN-1\}$, convergent in some punctured, branched, neighbourhood of $z_\infty$.
\end{theorem}
\begin{remark}
\label{remark1}
As mentioned above, the existence of the formal series solutions (\ref{formalseries}) is an assumption on the form of the equations, each formal series being equivalent to a differential relation between the coefficients $\alpha_{ij}(z)$. Whereas the existence of the formal series is clearly necessary for the solution to be represented by (\ref{solutionexpansion}) near a movable singularity, the theorem states that every movable singularity is of this form. If $(N+1) \nmid (MN-1)$, i.e. $d=1$, then there is really only one leading order behaviour for the solution (\ref{solutionexpansion}) as the choice of branch for $c_{1,-N-1}$ can be absorbed into the choice of branch for $(z-z_\infty)^{1/(MN-1)}$. In general there will be $d$ possible leading order behaviours.
\end{remark}

We assume in the following and for the rest of the article that $N \geq M$. In the neighbourhood of any movable singularity one can let
\begin{equation*}
\begin{aligned}
\tilde{y}_1(z) =& \left(\alpha_{M+1,0}(z)^N\alpha_{0,N+1}(z)\right)^{\frac{1}{MN-1}} \left( y_1(z) + \frac{\alpha_{M,0(z)}}{\alpha_{M+1,0}(z)} \right), \\
\tilde{y}_2(z) =& \left(\alpha_{M+1,0}(z)\alpha_{0,N+1}(z)^M\right)^{\frac{1}{MN-1}} \left( y_2(z) + \frac{\alpha_{0,N}(z)}{\alpha_{0,N+1}(z)} \right),
\end{aligned}
\end{equation*}
to achieve that the transformed Hamiltonian $\tilde{H}$ is of the same form as in (\ref{Hamiltonian}) but with $\tilde{\alpha}_{M+1} \equiv 1 \equiv \tilde{\alpha}_{0,N+1}$ and $\tilde{\alpha}_{0N} \equiv 0$ (and also $\tilde{\alpha}_{M0} \equiv 0$ if $N=M$). In the following we will assume that the Hamiltonian is already given in this normalised form and readily omit the tildes again,
\begin{equation}
\label{hamiltonian}
H(z,y_1,y_2) = y_1^{M+1} + y_2^{N+1} + \sum_{(i,j) \in I'} \alpha_{ij}(z)y_1^iy_2^j,
\end{equation}
where $I' = I \setminus \{(0,N)\}$, the Hamiltonian equations being 
\begin{equation}
\label{hamsystem}
\begin{aligned}
y_1' =& (N+1)y_2^N+ \sum_{(i,j) \in I'} j \alpha_{ij}(z) y_1^i y_2^{j-1}, \\
y_2' =& -(M+1)y_1^M - \sum_{(i,j) \in I'} i \alpha_{ij}(z) y_1^{i-1} y_2^j.
\end{aligned}
\end{equation}
For $N \geq M$, condition (\ref{indexcondition}) in fact implies that $j \leq N-1$ for all $(i,j) \in I'$.

\section{Preliminary lemmas}
\noindent We will make repeated use of the following lemma by Painlev\'e, see e.g. \cite{hille}.
\begin{lemma}
\label{painleve}
Let $F_k(z,y_1,\dots,y_m)$, $k=1,\dots,m$, be analytic functions in a neighbourhood of a point $(z_\infty,\eta_1,\dots,\eta_m) \in \mathbb{C}^{m+1}$. Let $\gamma$ be a curve with end point $z_\infty$  and suppose that $(y_1,\dots,y_m)$ are analytic on $\gamma \setminus \{z_\infty\}$ and satisfy
\begin{equation*}
y'_k = F_k(z,y_1,\dots,y_m), \quad k=1,\dots,m.
\end{equation*}
Suppose there is a sequence $(z_n)_{n \in \mathbf{N}} \subset \gamma$ such that $z_n \to z_\infty$ and $y_k(z_n) \to \eta_k \in \mathbf{C}$ as $n \to \infty$ for all $k=1,\dots,n$. 

Then the solution can be analytically continued to include the point $z_\infty$.
\end{lemma}
\begin{proof}
We can choose some $r$ such that all the functions $F_k$, $k=1,\dots,m$ are analytic in the set $D = \{ |z-z_\infty| \leq r, |y_k - \eta_k| \leq r, k=1,\dots,m \}$ and let $M = \max\{|F_k(z,y_1,\dots,y_m)| : (z,y_1,\dots,y_m) \in D, k = 1,\dots,m \}$. From some $n$ onwards, $\{|z-z_n| < r/2, |y_k-y_k(z_n)| < r/2, k=1,\dots,m \} \subset D$. By Cauchy's local existence and uniqueness theorem, a solution around $z_n$ is defined at least in the disc of radius $\rho = \frac{r}{2} \left( 1 - e^{-\frac{1}{(m+1)M}} \right)$. For some $n$ we have $z_\infty \in B(z_n,\rho)$.
\end{proof}

\noindent The next lemma is needed to show that an auxiliary function $W$, which will be constructed from the Hamiltonian $H$ in section \ref{firstintegral}, is bounded along $\gamma$. $W$ will be shown to satisfy a first-order linear differential equation of the form (\ref{lineardeq}) below. The lemma converts this into an integral representation for $W$.

\begin{lemma}
\label{firstorderlinear}
Let $\gamma$ be a finite length curve in the complex plane and let $P(z)$, $Q(z)$ and $R(z)$ be bounded functions on $\gamma$. Then any solution of the equation
\begin{equation}
\label{lineardeq}
W' = PW + Q + R',
\end{equation}
is also bounded on $\gamma$.
\end{lemma}
\begin{proof}
Choosing a point $z_0 \in \gamma$ the solution can be written as
\begin{equation*}
W(z) = R(z) + I(z) \left( C + \int_{z_0}^z (Q(\zeta)+P(\zeta)R(\zeta))I(\zeta)^{-1} d \zeta \right),
\end{equation*}
where $C=W(z_0) - R(z_0)$ is an integration constant and $I$ is the integrating factor
\begin{equation*}
I(z) = \exp \left( \int_{z_0}^z P(\zeta) d \zeta \right).
\end{equation*}
Since $P$, $Q$ and $R$ are bounded on $\gamma$ and $\gamma$ has finite length, $I(z)$ and $I(z)^{-1}$ are bounded and hence $W(z)$ itself is bounded on $\gamma$.
\end{proof}

\section{Curve modification}

In this section we will show that the curve $\gamma$ leading up to a singularity can be modified to a curve $\tilde{\gamma}$, still of finite length, such that it avoids the zeros of a solution $(y_1,y_2)$ of (\ref{hameqns}). This is a technical necessity to show that the auxiliary function $W$, to be constructed in section \ref{firstintegral}, is bounded on $\gamma$. The proof runs along the lines of a lemma by S. Shimomura \cite{shimomura1} in which he showed that for a solution $y(z)$ of a second order ODE of the form $y'' = E(z,y)(y')^2 + F(z,y) y' + G(z,y)$, by modifying a curve $\gamma$ ending in a singularity one can achieve that $y$ is bounded away on $\tilde{\gamma}$ from some fixed value $c$ for which the equation is non-singular.

Consider a differential system of two equations in $y_1$ and $y_2$ of the form
\begin{equation}
\label{system}
\begin{aligned}
y_1' = & F_1(z,y_1,y_2) \\
y_2' = & F_2(z,y_1,y_2)
\end{aligned}
\end{equation} 
where $F_1,F_2 \in \mathcal{O}_D[y_1,y_2]$ are polynomials in $y_1$, $y_2$ with coefficients analytic in some domain $D$ which we take to be a disc $D=\{z \in \mathbf{C} : |z-a| \leq R_0 \}$. We assume that $F_1,F_2$ are of the form 
\begin{equation}
\begin{aligned}
F_1(z,y_1,y_2) =& \alpha_{1 0 N_1} y_2^{N_1} + \sum_{j=0}^{M_1}\sum_{k=0}^{N_1-1} \alpha_{1jk}(z) y_1^j y_k^k, \\
F_2(z,y_1,y_2) =& \alpha_{2 M_2 0} y_1^{M_2} + \sum_{j=0}^{M_2-1}\sum_{k=0}^{N_2} \alpha_{2jk}(z) y_1^j y_k^k,
\end{aligned}
\end{equation}
where $N_1 \geq N_2$, $M_2 \geq M_1$ and $\alpha_{1 0 N_1}, \alpha_{2 M_2 0}$ are constants with $|\alpha_{1 0 N_1}| \geq 1$, $|\alpha_{2 M_2 0}| \geq 1$. Let $K>1$ be a constant so that $|\alpha_{ijk}(z)| < K$ for all $i,j,k$ and $z \in D$. Also, let $N_1:=N, M_2:=M$ and $C:=2^{N+1} (M+1)(N+1)K$.
\begin{lemma}
\label{circle}
Let $0 < \Delta < 1$ and $\theta := \min\{\frac{\Delta}{C},R_0\}$. Let $(y_1,y_2)$ be a solution of (\ref{system}) analytic at a point $c$ for which $|c-a| < \frac{R_0}{2}$. Suppose that $|y_1(c)| < \frac{\theta}{8}$ and $|y_2(c)| > C$. Then $(y_1(z),y_2(z))$ is analytic on the disc $|z-c| < \frac{\theta}{|y_2(c)|}$ and satisfies $|y_1(z)| \geq \frac{\theta}{8}$ and $|y_2(z)| \geq 1$ on the circle $|z-c| = \frac{\theta}{2 |y_2(c)|}$.
\end{lemma}
\begin{proof}
Let $\rho = y_2(c)^N$, $\zeta = \rho(z-c)$ and define $\eta_i(\zeta) := y_i(z)$, $i=1,2$. Denoting the derivative with respect to $\zeta$ by a dot we have $\dot{\eta}_i(\zeta) = \rho^{-1} y'_i(z)$ and
\begin{equation*}
\eta_i(\zeta) = \eta_i(0) + \int_0^\zeta \dot{\eta}_i(\tilde{\zeta}) d \tilde{\zeta},
\end{equation*}
where $\eta_i(0)=y_i(c)$. Define the functions $M_i(r) = \max_{|\zeta| \leq r} |\eta_i(\zeta)|$, $i=1,2$, and let $r_0 = \sup\{r : M_1(r) < \Delta, M_2(r) < 2|\rho|^{1/N}$\}. Clearly we have $r_0 > 0$. For $|\zeta| < \min\{r_0,R_0\}$ we have, since $|z-a| \leq |z-c| + |c-a| < \frac{R_0}{|\rho|} + \frac{R_0}{2} \leq R_0$,
\begin{equation}
\label{commoneta}
|\eta_i(\zeta)| \leq |y_i(c)| + |\rho|^{-1} |\zeta| \sum_{j=0}^{M_i} \sum_{k=0}^{N_i} K \Delta^j 2^k |\rho|^{\frac{k}{N}} \leq |y_i(c)| + |\zeta| 2^N K(N+1)(M+1).
\end{equation} 
Now suppose that $r_0 < \theta$. Then, for $|\zeta| < r_0 < R_0$ we have the estimates
\begin{equation*}
\begin{aligned}
|\eta_1(\zeta)| < & \theta (1/8+2^N(N+1)(M+1)K) < \Delta, \\
|\eta_2(\zeta)| < & |y_2(c)| + \theta 2^N(M+1)(N+1)K < 2|y_2(c)|,
\end{aligned}
\end{equation*}
in contradiction to the definition of $r_0$. Therefore we must have $r_0 \geq \theta$, showing that (\ref{commoneta}), $i=1,2$, is valid for $|\zeta| < \theta$ and therefore that $\eta_1$ and $\eta_2$ are analytic for $|\zeta| < \theta$. We now obtain estimates for $\eta_1$ and $\eta_2$ in the opposite direction on the circle $|\zeta| = \frac{\theta}{2}$:
\begin{equation*}
\begin{aligned}
|\eta_1(\zeta)| \geq & \left| \int_0^\zeta \rho^{-1} \alpha_{10N} \eta_2(\tilde{\zeta})^N \text{d} \tilde{\zeta} \right| -  \left| \int_0^\zeta \rho^{-1} \sum_{i=0}^{M_1} \sum_{j=0}^{N-1} \alpha_{1ij}(z) \eta_1^i \eta_2^j \text{d} \tilde{\zeta} \right| - |\eta_1(0)| \\
 \geq & \left| \int_0^\zeta \left(1 + \frac{\eta_2(\tilde{\zeta}) - \eta_2(0)}{\eta_2(0)}\right)^N \text{d} \tilde{\zeta} \right|  - \frac{\theta}{2} |\rho|^{-\frac{1}{N}} 2^{N-1}(M+1) N K - \frac{\theta}{8} \\
 \geq & \left| \int_0^\zeta \left( 1 + \sum_{n=1}^N \binom{N}{n} \left( \frac{\eta_2(\tilde{\zeta}) - \eta_2(0)}{\eta_2(0)} \right)^n \right) \text{d} \tilde{\zeta} \right| - \frac{\theta}{4} \\
 \geq & \frac{\theta}{2} - \frac{\theta}{2} \sum_{n=1}^N \binom{N}{n} \left(\frac{\Delta}{C}\right)^n - \frac{\theta}{4} \\
 \geq & \frac{\theta}{8}, \\
|\eta_2(\zeta)| \geq & |y_2(c)| -  \theta 2^N(M+1)(N+1)K \\
 \geq & 1.
\end{aligned}
\end{equation*}
\end{proof}
\begin{remark}
\label{interchange}
In Lemma \ref{circle} the role of $y_1$ and $y_2$ can be interchanged if in every expression one simultaneously replaces $M \leftrightarrow N$.
\end{remark}
Using Lemma \ref{circle} and Remark \ref{interchange} we can now show that a curve ending in a movable singularity of a solution $(y_1,y_2)$ of the system (\ref{system}) can be modified by arcs of circles in such a way that both $y_1$ and $y_2$ are bounded away from $0$ on the modified curve. The argument is very similar to the one in \cite{shimomura1}.
\begin{lemma}[1st curve modification]
\label{firstmodification}
Suppose $(y_1,y_2)$ is a solution of (\ref{system}), analytic on a finite length curve $\gamma \subset D$ up to, but not including its endpoint $z_\infty \in D$. Then we can deform $\gamma$, if necessary, in the region where $(y_1,y_2)$ is analytic, to a curve $\tilde{\gamma}$, still of finite length, such that $y_1$ and $y_2$ are bounded away from $0$ on $\tilde{\gamma}$ in a neighbourhood of $z_\infty$.
\end{lemma}
\begin{proof}
Let $\gamma$ be parametrised by arclength such that $\gamma(0)=z_0$, $\gamma(l)=z_\infty$ where $l$ is the length of $\gamma$. Define the two sets
\begin{equation*}
S_i := \{ s: 0 < s < l \text{ and } |y_i(\gamma(s))| \leq \theta/8 \}, \quad i=1,2.
\end{equation*}
We assume that $\liminf_{s \to l^-} \min \{|y_1|,|y_2|\} = 0$, otherwise there is nothing to show. Therefore the union $S_1 \cup S_2$ contains values arbitrarily close to $l$. There now exists some number $0 < s_0 < l$ with the following two properties: (i) $S_1 \cap S_2 \cap [s_0,l) = \emptyset$, (ii) whenever $s \in S_i$, $s>s_0$, we have $|y_{3-i}(\gamma(s))| > C$. Namely, if this was not the case we could find a sequence $z_i = \gamma(s_i)$, $s_i \to l$, such that $(y_1(z_i),y_2(z_i))$ is bounded and hence, by Lemma \ref{painleve}, the solution could be analytically continued to $z_\infty$ in contradiction to the assumption. Denote $S = (S_1 \cup S_2) \cap [s_0,l)$ and let $s_1 = \inf \{ s \in S : s > s_0 \}$. Suppose that $s_1 \in S_i$ and let $r_1 = \frac{\theta}{2|y_{3-i}(\gamma(s_1))|}$. Lemma \ref{circle} now shows that that $y_1$ and $y_2$ are analytic for $|z-\gamma(s_1)| < 2 r_1$ and that $|y_i(z)| \geq \theta/8$ and $|y_{3-i}(z)| \geq 1$ on the circle $C_1 = \{z: |z-\gamma(s_1)| = r_1\}$. We now recursively define a sequence of points $s_n$ and circles $C_n$ with radii $r_n$ as follows: Let $s_{n+1} = \inf \{s \in S: s > s_n + r_n \}$. If $s_{n+1} \in S_i$ ($i=1 \text{ or } 2$), then let $r_{n+1} = \frac{\theta}{2|y_{3-i}(\gamma(s_n))|}$.

By Lemma \ref{circle}, for every circle $C_n$, $n=1,2,\dots$, we have $|y_1(z)|,|y_2(z)| \geq \frac{\theta}{8}$ for all $z \in C_n$. Also, $\sum_{n=1}^\infty r_n \leq \sum_{n=1}^\infty |s_{n+1}-s_n| \leq l$ which implies $r_n \to 0$ as $n \to \infty$. The centres $s_n$ of the circles accumulate at $z_\infty$: If this was not the case we would have $s_n \to s_\infty$ for some $s_\infty < l$, but then
\begin{equation*}
\lim_{n \to \infty} \max \{|y_1(\gamma(s_n)|,|y_2(\gamma(s_n)|\} \geq \lim_{n \to \infty} \frac{\theta}{2r_n} = \infty,
\end{equation*}
in contradiction to the fact that $(y_1(z),y_2(z))$ is analytic on $\gamma \setminus \{z_\infty\}$. We now define $\tilde{\gamma}$ in the following way. Suppose for convenience that $\gamma$ has no self-intersections (otherwise we could shorten $\gamma$ by omitting pieces between self-intersections). Let $\gamma_\text{ext}$ be an infinite non-intersecting extension of $\gamma$ such that $\gamma_\text{ext}(s) \to \infty$ for $s \to \pm \infty$ which divides the complex plane into parts $\mathbf{C}_+$ and $\mathbf{C}_-$ such that $\mathbf{C}_+$, $\gamma_\text{ext}$ and $\mathbf{C}_-$ are pairwise disjoint and $\mathbf{C}_+ \cup \gamma_\text{ext} \cup \mathbf{C}_- = \mathbf{C}$. Now let $D = \gamma \cup \bigcup_{n=1}^\infty D_n$ where $D_n = \{z: |z-\gamma(s_n)| \leq r_n\}$ and define $\tilde{\gamma} = \partial D \cap (\mathbf{C}_+ \cup \gamma_\text{ext})$. Then $(y_1,y_2)$ is analytic on $\tilde{\gamma}$ and $|y_1(z)|,|y_2(z)| \geq \frac{\theta}{8}$ for all $z \in \tilde{\gamma}$. Furthermore, $\tilde{\gamma}$ has length less than $(1+2\pi)l$. 
\end{proof}

We will now specialise the results obtained so far in this section to the Hamiltonian system (\ref{hamsystem}) which is of the form (\ref{system}) with $N_1=N$, $M_2=M$. Lemma \ref{firstmodification} is not quite enough to show that the auxiliary function $W$ in section \ref{firstintegral}, rational in $y_1$ and $y_2$, is bounded. We need to show that certain terms of the form $\frac{y_2^k}{y_1^l}$ are bounded. To do so we will apply a second curve modification where we can now make use of the fact that $y_1$ and $y_2$ are already bounded away from $0$ on $\gamma$. We rewrite the system of equations (\ref{hamsystem}) in the variables $u_1 = y_1 \cdot y_2^{-\frac{N+1}{M+1}}$ and $u_2 = y_2$ for some branch of $y_2^{\frac{1}{M+1}}$.

The system of equations in the variables $u_1$, $u_2$ becomes
\begin{equation}
\label{usystem}
\begin{aligned}
u_1' = & (N+1) u_2^{N-\frac{N+1}{M+1}}\left( 1 + u_1^{M+1} \right) + \sum_{(i,j) \in I'} \left( j + i\frac{N+1}{M+1} \right) \alpha_{ij} u_1^i u_2^{(i-1)\frac{N+1}{M+1}+j-1} \\
u_2' = & -(M+1) u_1^M u_2^{M\frac{N+1}{M+1}} - \sum_{(i,j) \in I'} i\alpha_{ij} u_1^{i-1} u_2^{(i-1)\frac{N+1}{M+1}+j}.
\end{aligned}
\end{equation}
Let $K>1$ be a constant such that $|i \alpha_{ij}(z)| < K$ and $\left|\left(j+i\frac{N+1}{M+1}\right)\alpha_{ij}(z)\right| < K$ for all $(i,j) \in \tilde{I} = I' \cup \{(M+1,0),(0,N+1)\}$, $z \in D$. As before let $C = 2^{N+1}K(M+1)(N+1)$. Suppose $(u_1(z),u_2(z))$ is a solution of (\ref{usystem}), corresponding to a solution $(y_1(z),y_2(z))$ of (\ref{hamsystem}) on a curve $\gamma$, which by Lemma \ref{firstmodification} we assume to be such that $y_1$ and $y_2=u_2$ are bounded away from $0$ on $\gamma$. The following Lemma is somewhat similar to Lemma \ref{circle}, the proof, however, requires some modifications.
\begin{lemma}
Let $0 < \Delta < 2^{-N-2}(N+1)^{-1} < 1$ and $\theta := \min\{ \frac{\Delta}{C}, R_0 \}$. Let $(u_1,u_2)$ be a solution of (\ref{usystem}) analytic at $c$ with $|c-a| \leq \frac{R_0}{2}$ and suppose that $|u_1(c)| < \frac{\theta}{8}$ and $|u_2(c)| > (4C)^{M+1}$. Then $(u_1(z),u_2(z))$ is analytic in the disc $|z-c| < \frac{\theta}{|u_2(c)|}$ and on the circle $|z-c| = \frac{\theta}{2|u_2(c)|}$ we have $|u_1(c)| \geq \frac{\theta}{8}$ and $|u_2(c)| \geq 1$. 
\end{lemma}
\begin{proof}
Let $\rho = u_2(c)^L$, where $L = N - \frac{N+1}{M+1} \leq N-1$. For $i=1,2$ let $\eta_i(\zeta) := u_i(z)$, where $\zeta = \rho(z-c)$, and define $M_i(r) = \max_{|\zeta| \leq r} |\eta_i(\zeta)|$, $m_i(r) = \min_{|\zeta| \leq r} |\eta_i(\zeta)|$. Let 
\begin{equation}
\label{r0def}
r_0 = \sup \left\{r : M_1(r) < \Delta, M_2(r) < 2|\rho|^{1/L}, m_2(r) > \frac{1}{2}|\rho|^{1/L} \right\},
\end{equation}
which is positive as $|\eta_1(0)| < \Delta$ and $|\eta_2(0)| = |\rho|^{1/L}$. We have
\begin{equation*}
\eta_i(\zeta) = \eta_i(0) + \int_0^\zeta \dot{\eta}_i(\zeta) d \zeta,
\end{equation*}
where $\eta_i(0) = u_i(c)$ and $\dot{\eta}_i(\zeta) = \rho^{-1} u_i'(z)$. For $|\zeta| < \min\{r_0,R_0\}$ we have, since $|z-a| \leq |z-c| + |c-a| < \frac{R_0}{|\rho|} + \frac{R_0}{2} < R_0$,
\begin{equation}
\label{estimate1}
\begin{aligned}
|\eta_1(\zeta)| \leq & |u_1(c)| + |\rho|^{-1} |\zeta| \sum_{(i,j) \in \tilde{I} \setminus \{(0,0)\}} K \Delta^i 2^{|(i-1)\frac{N+1}{M+1}+j-1|} |\rho|^{((i-1)\frac{N+1}{M+1}+j-1)/L} \\
\leq & |u_1(c)| + |\zeta| 2^N K(M+1)(N+1), \\
\end{aligned}
\end{equation}
\begin{equation}
\label{estimate2}
\begin{aligned}
|\eta_2(\zeta)| \leq & |u_2(c)| + |\rho|^{-1} |\zeta| \sum_{\stackrel{(i,j) \in \tilde{I}}{i \neq 0}} K \Delta^{i-1} 2^{|(i-1)\frac{N+1}{M+1}+j|} |\rho|^{((i-1)\frac{N+1}{M+1}+j)/L} \\
\leq & |u_2(c)| \left( 1 + |\zeta| 2^N K (M+1)(N+1) \right), \\
|\eta_2(\zeta)| \geq &|u_2(c)| \left( 1 - |\zeta| 2^N K (M+1)(N+1) \right),
\end{aligned}
\end{equation}
where we have used condition (\ref{indexcondition}) which implies $(i-1)\frac{N+1}{M+1}+j-1 \leq L$ for $(i,j) \in \tilde{I} \setminus \{(0,0)\}$ and therefore $\left| (i-1)\frac{N+1}{M+1}+j-1\right| \leq N$. Now supposing that $r_0 < \theta$ one would obtain the estimates
\begin{equation*}
\begin{aligned}
|\eta_1(\zeta)| \leq & \theta (1/8 + 2^N K(M+1)(N+1)) < \Delta, \\
|\eta_2(\zeta)| \leq & |u_2(c)| \left( 1 + \theta 2^N K (M+1)(N+1) \right) < 2|\rho|^{1/L}, \\
|\eta_2(\zeta)| \geq &|u_2(c)| \left( 1 - \theta 2^N K (M+1)(N+1) \right) > \frac{1}{2}|\rho|^{1/L},
\end{aligned}
\end{equation*}
in contradiction to the definition (\ref{r0def}) of $r_0$. Therefore we must have $r_0 \geq \theta$, implying that the estimates (\ref{estimate1}), (\ref{estimate2}) are valid for $|\zeta| < \theta$ and that $u_1,u_2$ are analytic for $|\zeta| < \theta$. On the circle $|\zeta| = \frac{\theta}{2}$ we now have
\begin{equation*}
\begin{aligned}
|\eta_1(\zeta)| \geq & (N+1) \left| \int_0^\zeta \rho^{-1} \eta_2(\tilde{\zeta})^L d \tilde{\zeta} \right| - \left| \int_0^\zeta \rho^{-1} (N+1) \eta_1^{M+1} \eta_2^{N-\frac{N+1}{M+1}} d \tilde{\zeta} \right| \\ & - \left| \int_0^\zeta \rho^{-1} \sum_{(i,j) \in I'} \left( j + i\frac{N+1}{M+1} \right) \alpha_{ij} \eta_1^i \eta_2^{(i-1)\frac{N+1}{M+1}+j-1} d \tilde{\zeta} \right| - |\eta_1(0)| \\
 \geq & (N+1) \left| \int_0^\zeta \left(1 + \frac{\eta_2(\tilde{\zeta}) - \eta_2(0)}{\eta_2(0)} \right)^L d \tilde{\zeta} \right| - \frac{\theta}{2} (N+1)\Delta^{M+1} 2^L \\ & - \frac{\theta}{2} |\rho|^{-\frac{1}{L(M+1)}} 2^N K (M+1)(N+1) - \frac{\theta}{8} \\
 \geq & \left| \int_0^\zeta d \tilde{\zeta} \right| - \left| \int_0^\zeta \left( \left(1 + \frac{\eta_2(\tilde{\zeta}) - \eta_2(0)}{\eta_2(0)} \right)^L - 1 \right) d \tilde{\zeta} \right| - \frac{\theta}{4} \\
\geq & \frac{\theta}{4} - \frac{\theta}{2} \sum_{n=1}^{N} \binom{N}{n} \left(\frac{\Delta}{4C}\right)^n \geq \frac{\theta}{8}, \\
|\eta_2(\zeta)| \geq & \frac{1}{2}|\rho|^{1/L} > 1.
\end{aligned}
\end{equation*}
\end{proof}

\begin{lemma}[2nd curve modification]
\label{boundedpowers}
Let $(y_1,y_2)$ be a solution of the system (\ref{hamsystem}), analytic on the finite length curve $\gamma$ ending in a movable singularity $z_\infty$, such that $\frac{1}{y_1}$ and $\frac{1}{y_2}$ are bounded on $\gamma$. Then, after a possible deformation of $\gamma$ in the region where $y_1,y_2$ are analytic, one can achieve that $\frac{y_2^k}{y_1^l}$ is bounded on $\tilde{\gamma}$ for all $k,l \geq 0$ for which $l(N+1)-k(M+1) \geq 0$.
\end{lemma}
\begin{proof}
Define the set $S = \{s: 0 < s < l \text{ and } |u_1(\gamma(s))| \leq \theta/8 \}$. There exists some $s_0$, $0 < s_0 < l$, such that on $S \cap [s_0,l]$ one has $|u_2(z)| > (4C)^{M+1}$. For, if this was not the case, one would have a sequence of points $(z_n)$ on $\gamma$ with $z_n \to z_\infty$ as $n \to \infty$ such that $u_1(z_n)$ is bounded and $u_2(z_n)$ is bounded and bounded away from zero. Lemma \ref{painleve} applied to the system (\ref{usystem}) would then imply that $u_1,u_2$ are analytic at $z_\infty$ in contradiction to the assumption. By the same method as in the proof of Lemma \ref{firstmodification} one can now deform the curve $\gamma$ by arcs of circles such that $u_1$ and $u_2$ are bounded away from $0$ on the modified curve $\tilde{\gamma}$, that is, $u_1^{-(M+1)} = \frac{y_2^{N+1}}{y_1^{M+1}}$ and $u_2^{-1} = \frac{1}{y_2}$ are bounded on $\tilde{\gamma}$. By writing 
\begin{equation*}
\frac{y_2^k}{y_1^l} = \left( \left( \frac{y_2^{N+1}}{y_1^{M+1}} \right)^l \cdot \frac{1}{y_2^{l(N+1)-k(M+1)}} \right)^{1/(M+1)},
\end{equation*}
one can conclude that $\frac{y_2^k}{y_1^l}$ is bounded on $\tilde{\gamma}$ if $l(N+1) - k(M+1) \geq 0$.
\end{proof}

\section{An approximate first integral}
\label{firstintegral}
In this section we will show the existence of a function $W$ that remains bounded whenever a solution $(y_1(z),y_2(z))$ develops a movable singularity by analytic continuation along a finite length curve. Formally inserting the series expansions (\ref{formalseries}) for $y_1$ and $y_2$ into 
\begin{equation}
\label{Hdiffz}
H' = \frac{dH}{dz} = \frac{\partial H}{\partial z} = \sum_{(i,j) \in I'} \alpha_{ij}'(z) y_1(z)^i y_2(z)^j,
\end{equation}
yields a formal series expansion for $H'$ in $(z-z_0)^{\frac{1}{MN-1}}$. Heuristically, $W$ is constructed from $H$ by adding certain terms, rational in $y_1$ and $y_2$, which would cancel all terms of $H'$ with negative powers of $(z-z_0)^{\frac{1}{MN-1}}$. Note, however, that terms of power $(z-z_0)^{-1}$ cannot be cancelled in this way, since these would correspond terms of $H$ that are logarithmic in $z-z_0$ and cannot be obtained by rational expressions in $y_1$ and $y_2$. We define 
\begin{equation}
\label{Wdef}
W(z,y_1,y_2) = y_1^{M+1} + y_2^{N+1} + \sum_{(i,j) \in I'} \alpha_{ij}(z)y_1^iy_2^j + \sum_{(k,l) \in J} \beta_{kl}(z)\frac{y_2^k}{y_1^l},
\end{equation}
where the $\beta_{kl}(z)$ are certain analytic functions to be determined in terms of the $\alpha_{ij}(z)$ and their derivatives, and the index set $J$ is given by 
\begin{equation*}
J = \{(k,l) \in \mathbf{N}^2: 1 \leq k \leq N+1, 1-MN < k(M+1) - l(N+1) < M+N+2 \}.
\end{equation*}
Note that the pairs of indices in the set $J$ are in one-to-one correspondence with the elements of the set $I \setminus \{(0,0)\}$, which can easily be seen by setting $k=j+1$ and $l=M-i$. Thus for every unbounded term $\alpha_{ij}'(z) y_1^i y_2^j$ in (\ref{Hdiffz}) there is one function $\beta_{kl}$ to compensate for. However, we will see that not all the functions $\beta_{kl}$ can be used. The other essential ingredient is the existence of the formal series solutions (\ref{formalseries}), which will ensure that the terms of power $(z-z_0)^{-1}$ vanish identically. We will now show formally that $W$ is bounded.

\begin{lemma}
\label{Wbounded}
The coefficients $\beta_{kl}(z)$, $(k,l) \in J$, in (\ref{Wdef}) can be chosen such that the function $W$ is bounded on the curve $\tilde{\gamma}$.
\end{lemma}
\begin{proof}
Taking the total $z$-derivative of (\ref{Wdef}) one obtains
\begin{equation}
\label{Wdiffz}
\begin{aligned}
W' =& \sum_{(i,j) \in I'} \alpha_{ij}' y_1^i y_2^j + \sum_{(k,l) \in J} \left( \beta_{kl}' \frac{y_2^k}{y_1^l} + k \beta_{kl} \frac{y_2^{k-1} y_2'}{y_1^l} - l \beta_{kl} \frac{y_2^k y_1'}{y_1^{l+1}} \right) \\
   =& \sum_{(i,j) \in I'} \alpha_{ij}' y_1^i y_2^j - \sum_{(i,j) \in I'} \sum_{(k,l) \in J} (ik+jl) \alpha_{ij} \beta_{kl} y_1^{i-l-1} y_2^{k+j-1} \\ & + \sum_{(k,l) \in J} \left( \beta_{kl}' \frac{y_2^k}{y_1^l} - k (M+1) \beta_{kl} y_1^{M-l} y_2^{k-1} - l (N+1) \beta_{kl} \frac{y_2^{N+k}}{y_1^{l+1}} \right) \\ 
   =& \sum_{(i,j) \in I'} \alpha_{ij}' y_1^i y_2^j + \sum_{(k,l) \in J} (l(N+1) - k(M+1)) \beta_{kl}y_1^{M-l}y_2^{k-1} \\ & + \sum_{(k,l) \in J} \left( \beta_{kl}' \frac{y_2^k}{y_1^l} - l(N+1)\beta_{kl}\frac{y_2^{k-1}}{y_1^{l+1}}W \right) \\ & + \sum_{(i,j) \in I'} \sum_{(k,l) \in J} (l(N-j+1)-ik) \alpha_{ij} \beta_{kl} y_1^{i-l-1} y_2^{k+j-1} \\ & + \sum_{(k,l) \in J} \sum_{(k',l') \in J} l(N+1) \beta_{kl} \beta_{k'l'} \frac{y_2^{k+k'-1}}{y_1^{l+l'+1}},
\end{aligned}
\end{equation}
where we have used (\ref{Wdef}). All terms in (\ref{Wdiffz}) are now either of the form $y_1^{i_0} y_2^{j_0}$ with $(i_0,j_0) \in I$, or of the form $\frac{y_2^{j_0}}{y_1^{i_0}}$ with $i_0 \geq 1$ and $j_0(M+1)-i_0(N+1)<(M+1)(N+1)$. Note also that for the coefficients $\frac{y_2^{k-1}}{y_1^{l+1}}$ of $W$, $(k,l) \in J$, we have $(l+1)(N+1) - (k-1)(M+1) \geq 0$, i.e. by Lemma \ref{boundedpowers} these are bounded on $\tilde{\gamma}$. By repeating the process of replacing powers $y_2^{N+1}$ using (\ref{Wdef}) one can achieve in a finite number of steps that the terms of the form $\frac{y_2^{j_0}}{y_1^{i_0}}$ either have $j_0 \geq N+1$ with $i_0(N+1) - j_0(M+1) \geq 0$ and are therefore bounded by Lemma \ref{boundedpowers}, or have $j_0 \leq N$ and $j_0(M+1)-i_0(N+1) \leq MN-1$, equality holding if and only if $(i_0,j_0) = (1,N)$. We now manipulate the terms of the form $\frac{y_2^{j_0}}{y_1^{i_0}}$, $j_0 \leq N$, in the following way
\begin{equation}
\label{ratioterms}
\begin{aligned}
(M+&1)(j_0+1) \frac{y_2^{j_0}}{y_1^{i_0}} \\
=& -(j_0+1)\frac{y_2' y_2^{j_0}}{y_1^{M+i_0}} - \sum_{(i,j) \in I'} i(j_0+1) \alpha_{ij} \frac{y_2^{j+j_0}}{y_1^{M-i+i_0+1}} \\
=& -\left( \frac{y_2^{j_0+1}}{y_1^{M+i_0}} \right)' - (M+i_0) \frac{y_2^{j_0+1}y_1'}{y_1^{M+i_0+1}} - \sum_{(i,j) \in I'} i(j_0+1) \alpha_{ij} \frac{y_2^{j+j_0}}{y_1^{M-i+i_0+1}} \\
=& -\left( \frac{y_2^{j_0+1}}{y_1^{M+i_0}} \right)' - (N+1)(M+i_0) \frac{y_2^{N+j_0+1}}{y_1^{M+i_0+1}} \\ & - \sum_{(i,j) \in I'} \left(i(j_0+1) + j(M+i_0)\right) \alpha_{ij} \frac{y_2^{j+j_0}}{y_1^{M-i+i_0+1}} \\
=& -\left( \frac{y_2^{j_0+1}}{y_1^{M+i_0}} \right)' - (N+1)(M+i_0) \frac{y_2^{j_0}}{y_1^{M+i_0+1}}W  \\ & + \sum_{(i,j) \in I'} \left( (N+1)(M+i_0) - j(M+i_0) - i(j_0+1) \right) \alpha_{ij} \frac{y_2^{j+j_0}}{y_1^{M-i+i_0+1}} \\ & + \sum_{(k,l) \in J} (N+1)(M+i_0) \beta_{kl} \frac{y_2^{k+j_0}}{y_1^{M+l+i_0+1}} + (N+1)(M+i_0) \frac{y_2^{j_0}}{y_1^{i_0}}.
\end{aligned}
\end{equation}
Thus, unless $j_0(M+1)-i_0(N+1) = MN-1$, one can solve (\ref{ratioterms}) for $\frac{y_2^{j_0}}{y_1^{i_0}}$:
\begin{equation}
\label{ratiosolved}
\begin{aligned}
\frac{y_2^{j_0}}{y_1^{i_0}} =& \frac{1}{MN-1+i_0(N+1)-j_0(M+1)} \left( (N+1)(M+i_0) \frac{y_2^{j_0}}{y_1^{M+i_0+1}} W \right. \\ & \left. + \sum_{(i,j) \in I'} \left( i(j_0+1) + j(M+i_0) - (N+1)(M+i_0) \right) \alpha_{ij} \frac{y_2^{j+j_0}}{y_1^{M-i+i_0+1}} \right. \\ & \left. - \sum_{(k,l) \in J} (N+1)(M+i_0) \beta_{kl} \frac{y_2^{k+j_0}}{y_1^{M+l+i_0+1}} + \left( \frac{y_2^{j_0+1}}{y_1^{M+i_0}} \right)' \right).
\end{aligned}
\end{equation}
Again, in (\ref{ratiosolved}) the coefficient $\frac{y_2^{j_0}}{y_1^{M+i_0+1}}$ of $W$ is bounded by Lemma \ref{boundedpowers} since we have $(M+i_0+1)(N+1) - j_0(M+1) > 0$. Also, the term $\frac{y_2^{j_0+1}}{y_1^{M+i_0}}$ is bounded by Lemma \ref{boundedpowers} since $(M+i_0)(N+1) - (j_0+1)(M+1) > 0$. Therefore, the term $\left( \frac{y_2^{j_0+1}}{y_1^{M+i_0}} \right)'$ is bounded when integrated over the finite length curve $\tilde{\gamma}$. For the terms of type $\frac{y_2^{k+j_0}}{y_1^{M+l+i_0+1}}$, $(k,l) \in J$, we find
\begin{equation*}
(M+l+i_0+1)(N+1) - (k+j_0)(M+1) \geq 0,
\end{equation*}
which are therefore all bounded, and for the terms $\frac{y_2^{j+j_0}}{y_1^{M-i+i_0+1}}$, $(i,j) \in I'$,
\begin{equation*}
(j+j_0)(M+1) - (M-i+i_0+1)(N+1) < j_0(M+1) - i_0(N+1).
\end{equation*}
We can thus replace $\frac{y_2^{j_0}}{y_1^{i_0}}$ by terms which are bounded or proportional to $W$ with bounded factor, and a sum of terms of the form $\frac{y_2^{j_1}}{y_1^{i_1}}$ with $j_1=j+j_0$, $i_1=M-i+i_0+1$, such that the quantity $j_1(M+1) - i_1(N+1)$ is strictly decreasing. Performing this process iteratively a finite number of times we eventually end up only with terms $\frac{y_2^{j_n}}{y_1^{i_n}}$ for which $j_n(M+1) - i_n(N+1) \leq 0$, Lemma \ref{boundedpowers} showing that they are bounded on $\tilde{\gamma}$.

We thus arrive at a first-order differential equation for $W$ of the form
\begin{equation*}
\begin{aligned}
W' =& P(z,y_1^{-1},y_2) W + \sum_{(i,j) \in I} \gamma_{ij}(z) y_1^i y_2^j + \gamma_{-1N}(z) \frac{y_2^N}{y_1} \\ & + Q(z,y_1^{-1},y_2) + \frac{d}{dz} R(z,y_1^{-1},y_2),
\end{aligned}
\end{equation*} 
where $P$, $Q$ and $R$ are polynomial in their last two arguments and for each monomial $\frac{y_2^k}{y_1^l}$ we have $l(N+1)-k(M+1) \geq 0$, i.e. they are bounded on $\tilde{\gamma}$. We will now show that, by a suitable choice of the $\beta_{kl}$ and the existence of the formal series solutions (\ref{formalseries}), all the coefficients $\gamma_{ij}$, $(i,j) \in I$, as well as $\gamma_{-1N}$, are identically $0$.

We determine the functions $\beta_{kl}=\beta_{j+1,M-i}$ recursively starting with the pairs $(i,j) \in I$ for which the quantity $i(N+1)+j(M+1)$ is maximal. From (\ref{Wdiffz}) we see that 
\begin{equation}
\label{gammaij}
\gamma_{ij}(z) = \alpha_{ij}'(z) + (MN -1 -i(N+1)-j(M+1))\beta_{j+1,M-i}(z) + \cdots,
\end{equation}
where the dots stand for expressions involving only terms $\beta_{k'l'}=\beta_{j'+1,M-i'}$ for which $i'(N+1)+j'(M+1)$ is strictly greater than $i(N+1)+j(M+1)$. We can thus determine $\beta_{kl} = \beta_{j+1,M-i}$ for all pairs $(i,j) \in I$ for which $i(N+1)+j(M+1)>MN-1$. However, when $i(N+1)+j(M+1)=MN-1$, the coeffcient of $\beta_{j+1,M-i}$ in (\ref{gammaij}) vanishes. We now make use of the existence of the formal series solutions (\ref{formalseries}) to show that also $\gamma_{ij} \equiv 0$ in this case.  

Let $n=\frac{N+1}{d}$ and $m=\frac{M+1}{d}$ where $d=\gcd\{M+1,N+1\}$. Consider the $d$ terms $\gamma_{-1,N}(z) \frac{y_2^N}{y_1}$, $\gamma_{m-1,N-n}(z) y_1^{m-1}y_2^{N-n},\dots,\gamma_{M-m,n-1}(z) y_1^{M-m}y_2^{n-1}$. When one inserts the formal series solutions (\ref{formalseries}) into these expressions they have leading order $(z-z_0)^{-1}$. But, as explained in Remark \ref{remark1}, there are essentially $d$ formal series solutions corresponding to the different choices of the leading coefficients $c_{1,-N-1},c_{2,-M-1}$ such that $c_{1,-N-1}^{MN-1} =-\frac{1}{(MN-1)^{N+1}}$. Inserting any of the series into (\ref{Wdef}) shows that $W$ has a Laurent series expansion in powers of $(z-z_0)^{1/(MN-1)}$. Therefore, the coefficient of $(z-z_0)^{-1}$ in $W'$ vanishes since otherwise $W$ would have logarithmic terms in its expansion. The coefficients of $(z-z_0)^{-1}$ in $W'$, for the different choices of $(c_{1,-N-1},c_{2,-M-1})$, are
\begin{equation*}
\begin{aligned}
\frac{-1}{MN-1} \left( \gamma_{-1,N}(z_0) + \omega_1 \gamma_{m-1,N-n}(z_0) + \cdots + \omega_1^{d-1} \gamma_{M-m,n-1}(z_0) \right) =& 0 \\
\frac{-1}{MN-1} \left( \gamma_{-1,N}(z_0) + \omega_2 \gamma_{m-1,N-n}(z_0) + \cdots + \omega_2^{d-1} \gamma_{M-m,n-1}(z_0) \right) =& 0 \\
\vdots \\
\frac{-1}{MN-1} \left( \gamma_{-1,N}(z_0) + \omega_d \gamma_{m-1,N-n}(z_0) + \cdots + \omega_d^{d-1} \gamma_{M-m,n-1}(z_0) \right) =& 0,
\end{aligned}
\end{equation*}
where $\omega_i$, $i=1,\dots,d$, are the $d$ distinct roots of $\omega^d = -1$. This system of $d$ equations shows
\begin{equation*}
 \gamma_{-1,N}(z_0) = \gamma_{m-1,N-n}(z_0) = \cdots = \gamma_{M-m,n-1}(z_0) = 0.
\end{equation*}
However, the formal series expansions exist for all $\hat{z}$ in a neighbourhood of $z_0$. Therefore we have shown in fact that
\begin{equation*}
 \gamma_{-1,N} = \gamma_{m-1,N-n} = \cdots = \gamma_{M-m,n-1} \equiv 0.
\end{equation*}
The functions $\beta_{j+1,M-i}$ with $i(N+1)+j(M+1)=MN-1$ can be chosen arbitrarily and will henceforth be set to $0$. The remaining functions $\beta_{j+1,M-i}$ with $i(N+1)+j(M+1)<MN-1$ can now all be determined recursively, so that $\gamma_{ij} \equiv 0$ for all $(i,j) \in I \cup \{(-1,N)\}$. We have thus arrived at a first-order linear differential equation for $W$ of the form
\begin{equation}
\label{finalWeq}
W' = P(z,y_1^{-1},y_2) W + Q(z,y_1^{-1},y_2) + R'(z,y_1^{-1},y_2),
\end{equation}
where $P$, $Q$ and $R$ are bounded on $\tilde{\gamma}$ near a movable singularity $z_0$ of a solution $(y_1(z),y_2(z))$. Lemma \ref{firstorderlinear} now shows that $W$ is bounded on $\tilde{\gamma}$. 
\end{proof}

\section{A regular initial value problem}
To show that a movable singularity is an algebraic branch point we will now introduce coordinates $u$ and $v$ for which there exists a regular initial value problem. The coordinate $u$ is defined by 
\begin{equation}
\label{udef}
y_1 = u^{-\frac{N+1}{d}},
\end{equation}
where a choice of branch is made. We also define
\begin{equation}
\label{wdef}
w=y_2 u^{\frac{M+1}{d}}.
\end{equation}
From (\ref{Wdef}) one obtains an algebraic equation for $w$,
\begin{equation}
\label{wequation}
\begin{aligned}
0 =& w^{N+1} + \sum_{(i,j) \in I'} \alpha_{ij}(z) u^{\frac{(M+1)(N+1)-i(N+1)-j(M+1)}{d}} w^j \\ & + \sum_{(k,l) \in J} \beta_{kl}(z)u^{\frac{(M+1)(N+1)+l(N+1)-k(M+1)}{d}} w^k + 1 - Wu^{\frac{(M+1)(N+1)}{d}},
\end{aligned}
\end{equation}
all the exponents of $u$ being positive integers. The solutions of this equation for $w$ will be denoted by $w_1,\dots,w_{N+1}$. They are analytic functions of $u$, $z$ and $W$ in some neighbourhood of $u=0$, $z=z_\infty$ and $W=W_0$ for any $W_0 \in \mathbf{C}$. We express the $w_n$ as power series in $u$ and $W$ with analytic coefficients in $z$,
\begin{equation*}
w_n = F_n(z,u,W) = \omega_n \sum_{j,k=0}^\infty a_{jkn}(z) u^j W^k,
\end{equation*}
where $\omega_n$, $n=1,\dots,N+1$, are the distinct roots of $\omega^{N+1} = -1$, $a_{00n} \equiv 1$, and the first monomial containing $W$ is of the form $- \frac{1}{N+1}u^{\frac{(M+1)(N+1)}{d}}W$. We denote $\bar{F}_n(z,u) = \sum_{j=0}^{\frac{(M+1)(N+1)}{d}} a_{j0n}(z) u^j$ and define functions $v_n$ by 
\begin{equation}
\label{vdefined}
w_n = \omega_n \left( \bar{F}_n(z,u) - \frac{1}{N+1} u^{\frac{(M+1)(N+1)}{d}}v_n \right),
\end{equation}
so that in the limit $u \to 0$, $v_n$ agrees to leading order with $W$. From the definiton (\ref{wdef}) of $w$ we see that the choice of branch for $\omega_n$ can partially be absorbed into the original choice of branch for $u$ if $1 < d < M+1$, and completely be absorbed if $d=1$, so that there are essentially only $d$ inequivalent choices for $(u,v_n)$. From (\ref{udef}) and (\ref{hamsystem}) we obtain the differential equation satisfied by $u$:
\begin{equation}
\label{uequation}
\begin{aligned}
u' =& -\frac{d}{N+1} u^{\frac{N+1}{d}+1} \Bigg[ (N+1) \omega_n^N \left(u^{-\frac{M+1}{d}} \bar{F}_n(z,u) - \frac{1}{N+1} u^{\frac{(M+1)N}{d}}v_n \right)^N \\ & + \sum_{(i,j) \in I'} j\alpha_{ij}(z)u^{-i\frac{N+1}{d}} \omega_n^{j-1} \left(u^{-\frac{M+1}{d}} \bar{F}_n(z,u) - \frac{1}{N+1} u^{\frac{(M+1)N}{d}}v_n \right)^{j-1} \Bigg].
\end{aligned}
\end{equation}
Taking the reciprocal of (\ref{uequation}) and changing the role of the dependent and independent variables $u$ and $z$ one obtains, extracting the highest power of $u$ on the right hand side, an equation of the form,
\begin{equation}
\label{zequation}
\frac{dz}{du} = u^{\frac{MN-1}{d}-1} A(u,z,v),
\end{equation}
where $A(u,z,v)$ is analytic in $(u,z,v)$ at $(0,z_\infty,v_0)$ for any $v_0 \in \mathbf{C}$, and $A(0,z_\infty,v_0) = \frac{\omega_n}{d}$. We drop the index $n$ from now on. Reinserting (\ref{vdefined}) into (\ref{wequation}) yields an expression for $W$ in terms of $u$ and $v$ of the form
\begin{equation}
\label{Wexpressedbyv}
W = v + G(z,u,v),
\end{equation} 
where $G$ is a polynomial in $v$ of degree $N+1$ and analyic in $z$ and $u$ near $u=0$, satisfying $G(z,0,v) = 0$. We differentiate (\ref{Wexpressedbyv}) with respect to $z$,
\begin{equation}
\begin{aligned}
\label{Wprime}
W' =& v' + G_z + G_u u' + G_v v',
\end{aligned}
\end{equation}
and compare this with equation (\ref{finalWeq}), which can be written in the form
\begin{equation}
\label{firstorderWuv}
\begin{aligned}
W' =& \tilde{P}(z,u,v) W + \tilde{Q}(z,u,v) + \frac{d}{dz} \tilde{R}(z,u,v) \\
=& \tilde{P} (v + G) + \tilde{Q} + \tilde{R}_z + \tilde{R}_u u' + \tilde{R}_v v',
\end{aligned}
\end{equation}
where $\tilde{P}$, $\tilde{Q}$ and $\tilde{R}$ are polynomial in $u$ and $v$. One can solve (\ref{Wprime}) and (\ref{firstorderWuv}) for $v'$ to obtain an equation of the form
\begin{equation}
\label{vequation}
v' = B(z,u,v) u' + C(z,u,v),
\end{equation}
where $B$ and $C$ are analytic in their arguments. Multiplying (\ref{vequation}) by (\ref{zequation}) one obtains and equation for $v$ as function of $u$:
\begin{equation}
\label{vuequation}
\frac{dv}{du} = \frac{dv}{dz} \frac{dz}{du} = B(z,u,v) + u^{\frac{MN-1}{d}-1}A(z,u,v)C(z,u,v).
\end{equation}
Equations (\ref{zequation}) and (\ref{vuequation}) together form a regular initial value problem for $z$ and $v$ as functions of $u$ near $u=0$ with $z(0)=z_\infty$ and $v(0)=v_0$.

\section{Proof of Theorem 1}
We can now complete the proof of Theorem \ref{maintheo}.
\begin{proof}
By Lemma \ref{Wbounded}, after a possible modification of $\gamma$, the auxiliary function $W$ is bounded along $\gamma$. Consider a sequence $(z_n) \subset \gamma$ such that $z_n \to z_\infty$ as $n \to \infty$. Suppose that the sequence $(y_1(z_n))$ is bounded. Then the functional form of $W(z,y_1,y_2)$ implies that the sequence $(y_2(z_n))$ is also bounded. However, Lemma \ref{painleve} now implies that the solution $(y_1,y_2)$ can be analytically continued to $z_\infty$, in contradiction to the assumption in the theorem. Therefore, the sequence $(y_1(z_n))$ must tend to infinity since otherwise it would have a bounded subsequence. In the coordinates $u,v$ introduced in the previous section we therefore have that $u(z_n) \to 0$ and $v(z_n)$ is bounded. Hence there exists some subsequence $(z_{n_k})$ such that $v(z_{n_k}) \to v_0$ for some $v_0 \in \mathbf{C}$. Equations (\ref{zequation}) and (\ref{vuequation}) now form a regular initial value problem for $z$ and $v$ as functions of $u$ with initial values $z_\infty$ and $v_0$ at $u=0$. Lemma \ref{painleve} then shows that $z$ and $v$ are analytic at $u=0$. Since $A(0,z_\infty,v_0) \neq 0$ in (\ref{zequation}), $z$ has a convergent power series expansion of the form
\begin{equation*}
z = z_\infty + \sum_{k=0}^\infty \xi_k u^{k+\frac{MN-1}{d}},
\end{equation*}
in a neighbourhood of $u=0$. Taking the $\frac{MN-1}{d}$-th root,
\begin{equation*}
(z-z_\infty)^{\frac{d}{MN-1}} = \sum_{k=1}^\infty \eta_k u^k,
\end{equation*}
and inverting the power series, one shows that $u$ has a convergent series expansion
\begin{equation*}
u = \sum_{k=1}^\infty \zeta_k (z-z_\infty)^{\frac{kd}{MN-1}}.
\end{equation*}
By the definition (\ref{udef}) of $u$, one obtains a series expansion for $y_1$,
\begin{equation*}
y_1(z) = \sum_{k=-\frac{N+1}{d}}^\infty C_{1,k} (z-z_\infty)^{\frac{kd}{MN-1}},
\end{equation*}
convergent in a branched, punctured neighbourhood of $z_\infty$. Also, from the definition (\ref{wdef}) we find, since $w \neq 0$ at $z=z_\infty$,
\begin{equation*}
y_2(z) = \sum_{k=-\frac{M+1}{d}}^\infty C_{2,k} (z-z_\infty)^{\frac{kd}{MN-1}}.
\end{equation*}
\end{proof}

\section{Lowest degree examples}
If $M=1$ the Hamiltonian (\ref{hamiltonian}) can essentially be reduced to the form
\begin{equation*}
H(z,y_1,y_2) = \frac{1}{2} y_1^2 + P(z,y_2).
\end{equation*}
The Hamiltonian system thus corresponds to the second-order differential equation $y'' = P_y(z,y)$, the case of which was treated in \cite{halburd1}. This case includes the Painlev\'e equations $P_I$ (for $N=2$) and $P_{II}$ (for $N=3$). For $N \geq 4$ the equation has genuinely branched solutions. Let us now consider Hamiltonian systems where both $M,N \geq 2$.
\subsection{Case $M=N=2$}
The Hamiltonian here is of the form 
\begin{equation*}
H(z,y_1,y_2) = \frac{1}{3} y_1^3 + \frac{1}{3} y_2^3 + \alpha(z) y_1 y_2 + \beta(z) y_1 + \gamma(z) y_2,
\end{equation*}
where we have chosen a slightly different normalisation than the one in (\ref{hamiltonian}). The resonance conditions in this case are $\alpha'' \equiv 0$, $\beta' \equiv 0$ and $\gamma' \equiv 0$. One is therefore essentially left with
\begin{equation*}
H(z,y_1,y_2) = \frac{1}{3} y_1^3 + \frac{1}{3} y_2^3 + z y_1 y_2 + \beta y_1 + \gamma y_2,
\end{equation*}
the corresponding system of differential equations being
\begin{equation}
\label{painlevesystem}
\begin{aligned}
y_1' = & y_2^2 + z y_1 + \gamma \\
y_2' = & -y_1^2 - z y_2 - \beta.
\end{aligned}
\end{equation}
About any movable singularity $z_\infty$ a solution is represented by 
\begin{equation*}
y_1(z) = \sum_{k=-1}^\infty C_{1,k} (z-z_\infty)^k, \quad y_2(z) = \sum_{k=-1}^\infty C_{2,k} (z-z_\infty)^k,
\end{equation*}
with $C_{1,-1}^3 = -1$ and $C_{2,-1} = C_{2,-1}^2$, i.e. there are three possible leading order behaviours about any movable singularity which in this case are simple poles. Theorem \ref{maintheo} in this case states that every local solution $(y_1,y_2)$ extends to meromorphic functions in the whole complex plane, i.e. the system (\ref{painlevesystem}) has the Painlev\'e property. It is therefore of interest how its solutions can be expressed in terms of the six Painlev\'e transcendents. Therefore we let $y=y_1$ and eliminate $y_2$ from (\ref{painlevesystem}). This yields the following scalar differential equation of second order and second degree in $y$,
\begin{equation}
\label{scalareqn}
\left( y'' + zy' - (1-2z^2)y - 2\gamma z \right)^2 = 4\left(y^2+\beta\right)^2\left(y'-zy-\gamma\right).
\end{equation}
A Painlev\'e type classification for equations of second order and second degree has been done by C. Cosgrove and G. Scoufis in \cite{cosgrove1}. They found six inequivalent types of equations in the class $(y'')^2=F(z,y,y')$ which they denoted by SD-I -- SD-VI. All of these can be solved in terms of the Painlev\'e transcendents $P_I$ -- $P_{VI}$. In fact, equation (\ref{scalareqn}) is of the modified form denoted by SD-IV'.A (equation $5.87$ in \cite{cosgrove1}), which is solved in terms of $P_{IV}$.

\subsection{Case $M=2$, $N=3$}
In this case the normalised Hamiltonian (\ref{hamiltonian}) is
\begin{equation*}
H = y_1^3 + y_2^4 + \alpha_{21} y_1^2y_2 + \alpha_{12} y_1y_2^2 + \alpha_{11} y_1y_2 + \alpha_{20} y_1^2 + \alpha_{02} y_2^2 + \alpha_{10} y_1 + \alpha_{01} y_2.
\end{equation*}
The only resonance condition is
\begin{equation*}
\left( 3 \alpha_{12} - \alpha_{21}^2 \right)'' = 0,
\end{equation*}
and if it satisfied the solutions near a movable singularity $z_\infty$ are given by
\begin{equation*}
y_1(z) = \sum_{k=-4}^\infty C_{1,k} (z-z_\infty)^\frac{k}{5}, \quad y_2(z) = \sum_{k=-3}^\infty C_{2,k} (z-z_\infty)^\frac{k}{5},
\end{equation*}
with $C_{1,-4}^5 = -5^{-4}$, $C_{2,-3} = 5 C_{1,-4}^3$, where the choice for $C_{1,-4}$ can completely be absorbed into the choice of branch for $(z-z_\infty)^{\frac{1}{5}}$.

\subsection{Case $M=N=3$}
The normalised Hamiltonian is given by
\begin{equation*}
H = y_1^4 + y_2^4 + \alpha_{21} y_1^2y_2 + \alpha_{12} y_1y_2^2 + \alpha_{20} y_1^2 + \alpha_{11} y_1y_2 + \alpha_{02} y_2^2 + \alpha_{10} y_1 + \alpha_{01} y_2.
\end{equation*}
In order for the solutions to have only movable algebraic singularities the following conditions need to be satisfied,
\begin{equation*}
\left( 2 \alpha_{20} - \alpha_{12}^2 \right)' = 0, \quad \alpha_{11}' = 0, \quad \left( 2 \alpha_{02} - \alpha_{21}^2 \right)' = 0.
\end{equation*}
The solutions are given by
\begin{equation*}
y_1(z) = \sum_{k=-1}^\infty C_{1,k}(z-z_\infty)^{\frac{k}{2}}, \quad y_2(z) = \sum_{k=-1}^\infty C_{2,k}(z-z_\infty)^{\frac{k}{2}},
\end{equation*}
about any movable singularity $z_\infty$, where $C_{1,-1}^8 = - \frac{1}{16}$, $C_{2,-1} = 2 C_{1,-1}^3$, the choice for $C_{1,-1}$ however can only partially be absorbed into the choice of branch for $(z-z_\infty)^\frac{1}{2}$, i.e. there are $4$ possible leading order behaviours of the solution near any movable singularity.

\section{Summary and Outlook}
\noindent For a class of Hamiltonian systems of ordinary differential equations we have found that the only movable singularities obtained by analytic continuation along finite length curves are algebraic branched points, in particular these singularities are isolated and the solutions are locally finitely branched. The possibility of movable singularities obtained by analytic continuation along an infinite length curve is discussed by R. Smith in \cite{smith} for certain second-order differential equations. There it is shown that a singularity of this type is non-isolated, more specifically it is an accumulation point of algebraic singularities, and they cannot be ruled out at this stage for the systems presented here. It remains an interesting task to classify the structure of movable singularities for wider classes of differential equations. The author would like to express his sincere gratitude to Prof.\ R.\ Halburd for his invaluable support and many interesting discussions.

\bibliographystyle{plain}

\begin{thebibliography}{99}
\bibitem{cosgrove1} C. M. Cosgrove and G. Scoufis, {\it Painlev\'e classification of a class of differential equations of the second order and second degree}, Stud. Appl. Math. {\bf 88} (1993), 25--87
\bibitem{halburd1} G. Filipuk and R. G. Halburd, {\it Movable algebraic singularities of second-order ordinary differential equations}, J. Math. Phys. {\bf 50} (2009), 023509
\bibitem{garnier1} R. Garnier, {\it Sur des syst\'emes diff\'erentiels du second ordre dont l'int\'e-grale g\'en\'erale est uniforme}, Ann. Sci. \'Ecole Norm. Sup. {\bf 77} (1960), 123--144
\bibitem{goffar1} J. Goffar-Lombet, {\it Sur des syst\'emes polynomiaux d'\'equations diff\'eren-tielles dont l'int\'egrale g\'en\'erale est \'a points critiques fixes}, Acad. Roy. Belg. Cl. Sci. Mem. {\bf 41} (1974), 1--76
\bibitem{gromak} Gromak, V. I. and Laine, I. and Shimomura, S. {\it Painlev\'e differential equations in the complex plane}, De Gruyter Studies in Mathematics, Berlin, 2002
\bibitem{hille} E. Hille, {\it Ordinary {D}ifferential {E}quations in the {C}omplex {D}omain}, Wiley-Interscience, New-York, 1976
\bibitem{hinkkanen1} A. Hinkkanen and I. Laine, {\it Solutions of the first and second {P}ainlev\'e equations are meromorphic}, J. Anal. Math. {\bf 79} (1999), 345--77
\bibitem{kimura1} T. Kimura and T. Matuda, {\it On systems of differential equations of order two with fixed branch points}, Proc. Japan Acad. Ser. A Math. Sci. {\bf 56} (1980), 445--449
\bibitem{kimura2} T. Kimura, {\it Sur les points singuliers essentiels mobiles des \'equations diff\'erentielles du second ordre}, Comment. Math. Univ. St. Paul. {\bf 5} (1956), 81--94
\bibitem{malmquist2} J. Malmquist, {\it Sur les \'equations diff\'erentielles du second ordre, dont l'int\'egrale g\'en\'erale a ses points critiques fixes}, Ark. f\"or Mat., Astron. och Fys. {\bf 17} (1923), 1--89
\bibitem{murata1} Y. Murata, {\it On fixed and movable singularities of systems of rational differential equations of order {$n$}}, J. Fac. Sci. Univ. Tokyo Sect. IA Math. {\bf 35} (1988), 439--506
\bibitem{okamoto1} K. Okamoto, {\it Studies on the {P}ainlev\'e equations. {I}. {S}ixth {P}ainlev\'e equation {$P_{{\rm VI}}$}}, Ann. Mat. Pura Appl. {\bf 146} (1987), 337--381
\bibitem{okamoto2} K. Okamoto, {\it Studies on the {P}ainlev\'e equations. {II}. {F}ifth {P}ainlev\'e equation {$P_{\rm V}$}}, Japan. J. Math. {\bf 13} (1987), 47--76
\bibitem{okamoto3} K. Okamoto, {\it Studies on the {P}ainlev\'e equations. {III}. {S}econd and fourth {P}ainlev\'e equations, {$P_{{\rm II}}$} and {$P_{{\rm IV}}$}}, Math. Ann. {\bf 275} (1986), 221--255
\bibitem{okamoto4} K. Okamoto, {\it Studies on the {P}ainlev\'e equations. {IV}. {T}hird {P}ainlev\'e equation {$P_{{\rm III}}$}}, Funkcial. Ekvac. {\bf 30} (1987), 305--332
\bibitem{hukuhara1} K. Okamoto and K. Takano, {\it The proof of the {P}ainlev\'e property by {M}asuo {H}ukuhara}, Funkcial. Ekvac. {\bf 44} (2001), 201--217
\bibitem{painleve1} P. Painlev{\'e}, {\it M\'emoire sur les \'equations diff\'erentielles dont l'int\'egrale g\'en\'erale est uniforme}, Bull. Soc. Math. France {\bf 28} (1900), 201--261
\bibitem{shimomura1} S. Shimomura, {\it Proofs of the {P}ainlev\'e property for all {P}ainlev\'e equations}, Japan. J. Math. {\bf 29} (2003), 159--180
\bibitem{shimomura2} S. Shimomura, {\it A class of differential equations of {PI}-type with the quasi-{P}ainlev\'e property}, Ann. Mat. Pura Appl. {\bf 186} (2007), 267--280
\bibitem{shimomura3} S. Shimomura, {\it Nonlinear differential equations of second {P}ainlev\'e type with the quasi-{P}ainlev\'e property along a rectifiable curve}, Tohoku Math. J. {\bf 60} (2008), 581--595
\bibitem{smith} R. A. Smith, {\it On the singularities in the complex plane of solutions of $y'' + y' f(y) + g(y) = P(x)$}, Proc. London Math. Soc. {\bf 3} (1953), 498--512
\end{thebibliography}

\noindent Thomas Kecker \\
Department of Mathematics \\
University College London \\
Gower Street \\
London WC1E 6BT \\
United Kingdom \\
email: tkecker@math.ucl.ac.uk \\

\end{document}